\documentclass[a4paper,11pt,oneside]{article}
\usepackage[utf8]{inputenc}
\usepackage[T1]{fontenc}
\usepackage[english]{babel}
\usepackage{amsmath,amssymb,amsthm,mathtools,mathrsfs}
\usepackage{xcolor}
\usepackage{scrextend}
\usepackage{geometry}
\usepackage{enumitem}
\usepackage{hyperref}
\usepackage{comment}
\hypersetup{colorlinks=true,linkcolor=blue,citecolor=blue,urlcolor=blue}
\newtheorem{thm}{Theorem}[section]
\newtheorem{prop}[thm]{Proposition}
\newtheorem{lemma}[thm]{Lemma}
\newtheorem{coroll}[thm]{Corollary}
\newtheorem{remark}[thm]{Remark}
\newtheorem{defin}[thm]{Definition}
\newcommand{\OK}{\mathcal O_K}
\newcommand{\Fq}{\mathbb F_q}
\newcommand{\Q}{\mathbb Q}
\newcommand{\R}{\mathbb R}
\newcommand{\C}{\mathbb C}
\newcommand{\Z}{\mathbb Z}
\newcommand{\A}{\mathbb A}

\newcommand{\disc}{\operatorname{disc}}

\newcommand{\Norm}{\operatorname N}

\newcommand{\red}[1]{#1}
\title{}
\author{}
\date{}
\begin{document}

\begin{center}
\textbf{GALOIS GROUPS OF POLYNOMIALS OVER A FIXED NUMBER FIELD:\\
DETERMINANT METHODS IN GENERAL BOXES, RESOLVENTS, AND QUANTITATIVE COUNTING}
\end{center}
\begin{center}Ilaria Viglino\end{center}

\begin{center}\textbf{Abstract}\end{center}
\begin{addmargin}[2em]{2em}
Let $K$ be a fixed number field of degree $d$.  We prove two determinant-method results for integral points over $\mathcal O_K$
in boxes whose coordinates may have different size bounds: a Bombieri--Pila type
estimate for affine plane curves that retains the two separate coordinate bounds\cite{BombieriPila,BrowningHB},
and a Salberger--Browning covering theorem for affine surfaces\cite{Browning,Salberger}.  If $F\in\mathcal O_K[X_1,X_2]$ is
irreducible and
$T=\max B_1^{e_1}B_2^{e_2}$ over the monomials occurring in $F$, then
\[
 N_K(F;B_1,B_2)\ll T^\varepsilon
 \exp\!\left(d\frac{\log B_1\log B_2}{\log T}\right).
\]
For an absolutely irreducible surface in three variables we obtain a covering
by $O(T^\varepsilon V_{3,K})$ bounded-degree auxiliary hypersurfaces, outside
$O(T^\varepsilon V_{3,K}^2)$ exceptional integral points, where
\[
 V_{3,K}=\exp\!\left(d
 \left(\frac{\log B_1\log B_2\log B_3}{\log T}\right)^{1/2}\right).
\]
These are the number-field analogues of determinant-method estimates for unequal coordinate ranges developed in the rational setting by
Browning and Heath-Brown\cite{BrowningHB}, Salberger\cite{Salberger}, Browning\cite{Browning}, and Chow--Dietmann\cite{ChowDietmannVDW}.  We then apply them to
Galois resolvents of monic degree-$n$ polynomials over $K$.  Independently, a
large-sieve argument, in the tradition of Gallagher's probabilistic Galois theory\cite{Gallagher}, gives
$O_{K,n}(N^{dn-d/2}\log N)$ exceptional separable polynomials, while effective
Hilbert irreducibility over global fields\cite{ParedesSasykHIT} gives subgroup-sensitive bounds with exponent
$d(n-1+1/[S_n:G])$, up to logarithms.  We record the interaction with the
discriminant-square equation for even groups, congruence restrictions, and
rectangular coefficient boxes.  All determinant estimates used in the
applications are proved in the paper; no conjectural global-field determinant
statement is assumed.
\end{addmargin}

\section{Introduction}
Fix a number field $K/\Q$ of degree $d$ and an integral basis $\omega_1,\ldots,\omega_d$ of $\OK$.  For
\[
 \alpha=\sum_{j=1}^d a_j\omega_j\qquad(a_j\in\Z)
\]
put $H_{\mathcal B}(\alpha)=\max_j|a_j|$.  For $n\ge2$ and $N\ge1$, let
\[
 \mathscr P_{n,N}(K)=\left\{X^n+\alpha_{n-1}X^{n-1}+\cdots+\alpha_0:
 \alpha_i\in\OK,\ H_{\mathcal B}(\alpha_i)\le N\right\}.
\]
Thus $|\mathscr P_{n,N}(K)|=(2N+1)^{dn}$.  Write $G_f$ for the Galois group over $K$ of the splitting field of $f$, acting on the roots, and let
\[
 \mathscr P^0_{n,N}(K)=\{f\in\mathscr P_{n,N}(K):f\text{ separable and }G_f\cong S_n\}.
\]

The rational van der Waerden conjecture predicts, and Bhargava proved in full generality\cite{Bhargava}, that the complement of the $S_n$-polynomials in a rational coefficient box has codimension one.  Over a fixed number field the geometry of the box has dimension $dn$.  Our first contribution is instead analytic-geometric: we establish determinant estimates over $\mathcal O_K$ for boxes with different side lengths for curves and surfaces, with the precise exponents that record the different side lengths needed in resolvent counting.  We then combine these estimates with the large-sieve approach to probabilistic Galois theory\cite{Gallagher,Cohen} and effective Hilbert irreducibility over global fields\cite{ParedesSasykHIT}.

\subsection{Main results and contributions}

The principal new analytic contribution of this paper is a pair of integral-point estimates for boxes with different side lengths over a fixed number field.  They are the precise
number-field analogues of the two determinant-method inputs that underlie the
rational arguments of Browning--Heath-Brown\cite{BrowningHB}, Browning--Salberger\cite{Browning,Salberger}, and
Chow--Dietmann\cite{ChowDietmannVDW}.  To state them, for
\[
 F=\sum_{\mathbf e}c_{\mathbf e}X_1^{e_1}\cdots X_r^{e_r}\in
 \mathcal O_K[X_1,\ldots,X_r],\qquad B_i\ge2,
\]
write
\[
 T_F(\mathbf B)=\max_{c_{\mathbf e}\ne0}\prod_i B_i^{e_i}.
\]
We emphasize the elementary meaning of the boxes considered here.  A ``box
with different side lengths'' simply means that the coordinates are bounded
separately,
\[
 H_\infty(x_1)\le B_1,\quad \ldots,\quad H_\infty(x_r)\le B_r,
\]
with no assumption that the numbers $B_i$ are comparable.  The logarithms
$w_i=\log B_i$ are used as weights in the determinant argument.  Thus the
word ``weighted'' below refers only to keeping track of these different
coordinate bounds; it does not introduce an additional geometric hypothesis.

\begin{thm}[Curve estimate in a box with different side lengths; Theorem~\ref{thm:lopcurve} below]\label{thm:introcurve}
Let $F\in\mathcal O_K[X_1,X_2]$ be irreducible over $K$ of total degree at
most $D$.  Then, for every $\varepsilon>0$,
\[
 \#\{(x_1,x_2)\in\mathcal O_K^2:F(x_1,x_2)=0,
 H_\infty(x_i)\le B_i\}
 \ll_{K,D,\varepsilon}
 T_F(B_1,B_2)^\varepsilon
 \exp\!\left(d\frac{\log B_1\log B_2}{\log T_F(B_1,B_2)}\right).
\]
\end{thm}

\begin{thm}[Surface covering in a box with different side lengths; Theorem~\ref{thm:lopsurface} below]\label{thm:introsurface}
Let $F\in\mathcal O_K[X_1,X_2,X_3]$ define an absolutely irreducible affine
surface of degree at most $D$, put $T=T_F(B_1,B_2,B_3)$, and set
\[
 V_{3,K}=\exp\!\left(d
 \left(\frac{\log B_1\log B_2\log B_3}{\log T}\right)^{1/2}\right).
\]
Then all integral points of $F=0$ in the box $H_\infty(x_i)\le B_i$, except
for a set of size $O_{K,D,\varepsilon}(T^\varepsilon V_{3,K}^2)$, lie on
$O_{K,D,\varepsilon}(T^\varepsilon V_{3,K})$ hypersurfaces of bounded degree,
none containing $F=0$.
\end{thm}

The proofs occupy Section~\ref{sec:determinant} and
Appendix~\ref{app:weightedglobal}.  The essential new bookkeeping keeps track of the separate coordinate bounds:
the archimedean determinant is measured by the Newton support of $F$, while
finite-place divisibility is measured by ideal norms.  Taking the field norm
of the interpolation determinant is exactly what inserts the factor
$d=[K:\mathbb Q]$ in the exponential scales above.

\subsection{Proof strategy for the two determinant theorems}

Because Theorems~\ref{thm:introcurve} and~\ref{thm:introsurface} are the
main new analytic statements, we indicate at the outset where their proofs
require more than a formal replacement of $\mathbb Z$ by $\mathcal O_K$.
First, interpolation determinants are algebraic integers.  Their archimedean
size is estimated simultaneously at all embeddings of $K$, whereas their
finite-place divisibility is expressed by prime ideals; the comparison is made
by the absolute field norm $|N_{K/\mathbb Q}(\Delta)|$.  This is the source
of the factor $d=[K:\mathbb Q]$ in the final exponents.

Second, the different coordinate bounds are retained through a weighted
monomial order.  Writing $w_i=\log B_i$ and $t=\log T_F(\mathbf B)$, the
count and total weight of standard monomials modulo the leading monomial of
$F$ give the quantities
\[
 \exp\!\left(d\frac{w_1w_2}{t}\right)
 \quad\text{for curves},\qquad
 \exp\!\left(d\sqrt{\frac{w_1w_2w_3}{t}}\right)
 \quad\text{for surfaces}.
\]
These calculations are proved in Lemmas~\ref{lem:weightedcurve} and
\ref{lem:weightedsurface}; they are not inserted heuristically.

Third, uniformity in the coefficients of $F$ is treated by a
coefficient-height dichotomy.  Either the coefficient height is controlled by
the box parameter, or all points already lie on a second bounded-degree
hypersurface.  Fourth, prime ideals at which the reduction is reducible or
singular are excluded using integral Noether forms as in the global-field determinant-method framework\cite{ParedesSasykPoints}; their total contribution is
absorbed in the factor $T^\varepsilon$.

The curve theorem can then be completed with one good prime ideal and a local
interpolation determinant.  The surface theorem is genuinely different: using
one prime ideal would create about $(N\mathfrak p)^2$ residue classes and
would lose one full factor in the desired bound.  We therefore use the global
determinant construction, distributing divisibility over several prime ideals
and producing auxiliary hypersurfaces for many residue classes at once.  The
exceptional set is handled by the second stopping step of this construction.
Finally, throughout the surface argument we distinguish irreducibility over
$K$ from geometric irreducibility, and we separate nonlinear auxiliary curves,
linear components, and singular intersections before applying the curve
estimate.  Sections~\ref{sec:determinant} and~\ref{app:weightedglobal} give
these steps in detail and identify the precise external results used at each
point.

We then apply these estimates to Galois resolvents.  In parallel, we retain a
large-sieve theorem, which is useful because it applies uniformly to the union
of all exceptional Galois groups.  The resulting statements should be viewed
as complementary: the determinant method is sensitive to the Newton support
and to individual resolvents, whereas the large sieve gives a simple global
exceptional-set bound.

The large-sieve estimate is as follows; its proof adapts Gallagher's large-sieve method for Galois groups\cite{Gallagher}, together with standard number-field lattice counting.

\begin{thm}[Quantitative van der Waerden theorem over $K$]\label{thm:mainsep}
For fixed $K$ and $n\ge2$,
\[
 \#\{f\in\mathscr P_{n,N}(K):f\text{ separable and }G_f\ne S_n\}
 \ll_{K,n}N^{dn-d/2}\log(2N).
\]
\end{thm}

\begin{thm}[Full coefficient box]\label{thm:mainall}
For fixed $K$ and $n\ge2$,
\[
 \#\bigl(\mathscr P_{n,N}(K)\setminus\mathscr P^0_{n,N}(K)\bigr)
 \ll_{K,n}N^{dn-d/2}\log(2N)+N^{dn-1}.
\]
\end{thm}

The principal subgroup-sensitive refinement of this article is the following.  The
index $[S_n:G]$ measures the degree of the natural coset action of $S_n$ on
$S_n/G$ and is exactly the degree of the generic resolvent used below.

\begin{thm}[Fixed-subgroup bound]\label{thm:fixedG}
Let $G<S_n$ be a fixed subgroup and put $m=[S_n:G]$.  Then there exists
$\nu=\nu(n,G)\ge0$ such that
\[
 \#\{f\in\mathscr P_{n,N}(K): f\text{ separable and }
       G_f\subseteq gGg^{-1}\text{ for some }g\in S_n\}
 \ll_{K,n,G} N^{d(n-1+1/m)}(\log(2N))^{\nu}.
\]
In particular, the same estimate holds for polynomials with $G_f$ conjugate
to $G$.
\end{thm}

When $m>2$, Theorem~\ref{thm:fixedG} is strictly stronger than
Theorem~\ref{thm:mainsep}: the latter saves $d/2$ powers of $N$, whereas the
fixed-subgroup theorem saves $d(1-1/m)>d/2$.
For $G=A_n$, where $m=2$, the two exponents coincide.

For $n\ge5$ the classical low-index subgroup theorem for $S_n$ and $A_n$
gives a useful hierarchy.

\begin{coroll}[The exceptional hierarchy]\label{cor:hierarchy}
Assume $n\ge5$.  There is $\nu_n\ge0$ such that
\[
 \#\{f\in\mathscr P_{n,N}(K):f\text{ separable},\ G_f\ne S_n,
     \ G_f\not\subseteq A_n\}
 \ll_{K,n}N^{d(n-1+1/n)}(\log(2N))^{\nu_n},
\]
and
\[
 \#\{f\in\mathscr P_{n,N}(K):f\text{ separable},\ G_f\subsetneq A_n\}
 \ll_{K,n}N^{d(n-1+1/(2n))}(\log(2N))^{\nu_n}.
\]
Thus, at the $N^{d(n-1/2)}$ scale of the general exceptional set, the only
proper transitive group not forced into a smaller exponent by subgroup index
is $A_n$ itself.
\end{coroll}

\begin{coroll}[Reducible polynomials]\label{cor:reducible}
For $n\ge2$ there is $\nu_n\ge0$ such that the number of separable reducible
polynomials in $\mathscr P_{n,N}(K)$ is
\[
 \ll_{K,n}N^{d(n-1+1/n)}(\log(2N))^{\nu_n}.
\]
Including the nonseparable locus adds $O_{K,n}(N^{dn-1})$.
\end{coroll}

The rational theorem of Bhargava\cite{Bhargava} gives the sharper codimension-one bound
$O_n(N^{n-1})$ when $K=\mathbb Q$.  The purpose here is different: for a
fixed general number field we obtain a uniform quantitative theorem and a
subgroup-sensitive refinement.  The latter becomes increasingly strong as
the subgroup index grows.

\section{Determinant estimates in boxes with different side lengths over a number field}\label{sec:determinant}

This section proves the two weighted determinant estimates that are needed in the
Galois applications.  Over $\Q$ the corresponding curve estimate is the
unequal-side-length Bombieri--Pila estimate of Browning--Heath-Brown
\cite{BrowningHB}, while the surface covering statement is the refinement
recorded by Browning \cite[Lemma~1]{Browning} from Salberger's global
determinant method.  Our task is to keep the different coordinate bounds throughout
the global-field determinant argument.  The isotropic global-field machinery is
due to Paredes--Sasyk \cite{ParedesSasykPoints}; the affine global-field
formulation and its compatibility with the modern determinant method are also
explained by Vermeulen \cite[Section~6]{VermeulenAffine}.  We give the weighted
bookkeeping in full because it is precisely this part that is used later.

Throughout this section $K$ is a fixed number field of degree $d$, with ring of
integers $\OK$.  The phrase ``box with different side lengths'' will always
mean separate bounds $H_\infty(x_i)\le B_i$ with arbitrary $B_i\ge2$; no
comparability hypothesis such as $B_i\asymp B_j$ is imposed.  For an archimedean place $v$ put $d_v=[K_v:\R]$, so that
$\sum_{v\mid\infty}d_v=d$.  We use the multiplicative archimedean box height
\[
 H_\infty(x)=\max_{\sigma:K\hookrightarrow\C}|\sigma(x)|.
\]
Replacing this by the integral-basis height changes all side lengths by constants
depending only on $K$, by equivalence of norms on the fixed Minkowski lattice.

\subsection{The two main determinant theorems}

For
\[
 F(X_1,\ldots,X_r)=\sum_{\mathbf e}c_{\mathbf e}X_1^{e_1}\cdots X_r^{e_r}
 \in\OK[X_1,\ldots,X_r]
\]
and $B_i\ge2$, define
\begin{equation}\label{eq:weightedT}
 T_F(\mathbf B)=\max_{c_{\mathbf e}\ne0}\prod_{i=1}^rB_i^{e_i}.
\end{equation}
Notice that $T_F(\mathbf B)>1$ whenever $F$ is nonconstant.

\begin{thm}[Curve estimate in a box with different side lengths]\label{thm:lopcurve}
Let $D\ge1$ be fixed and let $F\in\OK[X_1,X_2]$ be irreducible over $K$,
of total degree at most $D$.  Put
\[
 N_K(F;B_1,B_2)=\#\{(x_1,x_2)\in\OK^2:F(x_1,x_2)=0,
 H_\infty(x_i)\le B_i\}.
\]
Then, for every $\varepsilon>0$,
\begin{equation}\label{eq:curve-main}
 N_K(F;B_1,B_2)
 \ll_{K,D,\varepsilon}
 T_F(B_1,B_2)^\varepsilon
 \exp\!\left(
 d\frac{\log B_1\log B_2}{\log T_F(B_1,B_2)}
 \right).
\end{equation}
The implied constant is uniform in the coefficients of $F$.
\end{thm}

\begin{thm}[Surface covering in a box with different side lengths]\label{thm:lopsurface}
Let $D\ge2$ be fixed and let $F\in\OK[X_1,X_2,X_3]$ define an absolutely
irreducible affine surface of degree at most $D$.  Write
\[
 T=T_F(B_1,B_2,B_3),\qquad
 V_{3,K}=\exp\!\left(
 d\left(\frac{\log B_1\log B_2\log B_3}{\log T}\right)^{1/2}
 \right).
\]
For every $\varepsilon>0$ there are polynomials
$G_1,\ldots,G_J\in\OK[X_1,X_2,X_3]$ and a finite set
$Z\subset\OK^3$ such that
\begin{align}
 J&\ll_{K,D,\varepsilon}T^\varepsilon V_{3,K},\label{eq:detJ}\\
 |Z|&\ll_{K,D,\varepsilon}T^\varepsilon V_{3,K}^2.\label{eq:detZ}
\end{align}
with the following properties:
\begin{enumerate}[label=\textup{(\roman*)}]
\item $\deg G_j=O_{K,D,\varepsilon}(1)$ and $F\nmid G_j$ for every $j$;
\item every $\mathbf x\in\OK^3\setminus Z$ satisfying
$F(\mathbf x)=0$ and $H_\infty(x_i)\le B_i$ lies on $G_j=0$ for some $j$.
\end{enumerate}
The degree bounds and implied constants are independent of the coefficients of
$F$.
\end{thm}

When $K=\Q$, Theorem~\ref{thm:lopcurve} recovers the curve estimate for unequal coordinate ranges used by
Chow--Dietmann\cite[Theorem~2.2]{ChowDietmannVDW}, and Theorem~\ref{thm:lopsurface} recovers Browning's surface-covering lemma\cite[Lemma~1]{Browning}, based on Salberger's global determinant method\cite{Salberger}.  The factor $d$ is necessary: for $X_2=X_1^r$ and $B_2=B_1^r$, the
left side of \eqref{eq:curve-main} has order $B_1^d$.

\subsection{Archimedean norms and determinants}

We first isolate the elementary number-field input.

\begin{lemma}[Norm comparison]\label{lem:normdet}
For $a\in\OK$,
\[
 |N_{K/\Q}(a)|=\prod_{v\mid\infty}|a|_v^{d_v}.
\]
In particular, if $a\ne0$ and an integral ideal $\mathfrak a$ divides $(a)$,
then
\[
 N\mathfrak a\le |N_{K/\Q}(a)|.
\]
\end{lemma}
\begin{proof}
The first identity is the product formula restricted to the archimedean
embeddings for an algebraic integer.  If $\mathfrak a\mid(a)$ then
$N\mathfrak a\mid |N_{K/\Q}(a)|$ as positive integers.
\end{proof}

Let $M_1,\ldots,M_R$ be monomials in $r$ variables, say
$M_j=X_1^{e_{1j}}\cdots X_r^{e_{rj}}$.  If $\mathbf x^{(1)},\ldots,
\mathbf x^{(R)}\in\OK^r$ lie in the box $H_\infty(x_i)\le B_i$, then
Hadamard's inequality at each embedding gives
\begin{equation}\label{eq:archdet}
 \prod_{v\mid\infty}
 \left|\det(M_j(\mathbf x^{(i)}))_{i,j}\right|_v^{d_v}
 \le R^{dR/2}\prod_{j=1}^R\prod_{i=1}^rB_i^{d e_{ij}}.
\end{equation}
The occurrence of $d$ in the exponent is the only archimedean change from the
rational argument.

\begin{lemma}[Local determinant divisibility]\label{lem:localdet}
Let $X/\OK$ be the closure of an integral affine variety of dimension $m$, and
let $\mathfrak p$ be a prime of good reduction.  Suppose
$\mathbf x^{(1)},\ldots,\mathbf x^{(R)}\in X(\OK)$ reduce to the same smooth
point of $X(\OK/\mathfrak p)$.  For any regular functions
$f_1,\ldots,f_R$ integral at that point,
\[
 \det(f_j(\mathbf x^{(i)}))_{i,j}\in\mathfrak p^{E_m(R)},
\]
where $\nu_0\le\nu_1\le\cdots$ is the nondecreasing sequence of
total degrees of monomials in $m$ variables (each degree repeated according
to its multiplicity) and
\[
 E_m(R)=\sum_{j=0}^{R-1}\nu_j
       =\frac{m}{m+1}(m!)^{1/m}R^{1+1/m}+O_m(R).
\]
For curves this reads $E_1(R)=R(R-1)/2$.
\end{lemma}
\begin{proof}
Complete the local ring at the smooth reduction point.  It is a formal power
series ring in $m$ parameters over the unramified completion of $\mathcal O_{K,\mathfrak p}$.
Ordering monomials in the parameters by total degree and successively subtracting
rows shows that the $j$th new row contributes at least the $\mathfrak p$-adic
order of the $j$th monomial.  Summing these orders gives $E_m(R)$.  This is the
standard local determinant lemma; the argument is purely local and is identical
to the global-field proof in \cite[Section~3]{ParedesSasykPoints}.
\end{proof}

The next elementary observation handles geometric reducibility in the curve
case.

\begin{lemma}[Geometrically reducible plane curves]\label{lem:georedcurve}
Let $F\in K[X,Y]$ be irreducible over $K$ but not absolutely irreducible, with
$\deg F\le D$.  Then the set of $K$-rational points on $F=0$ has cardinality
$O_D(1)$.
\end{lemma}
\begin{proof}
Over a finite Galois extension $L/K$, choose an absolutely irreducible factor
$F_1$ of $F$.  Since $F$ is irreducible over $K$, some conjugate $F_1^\sigma$ is
distinct from $F_1$.  Every $K$-point on $F=0$ lies on both $F_1=0$ and
$F_1^\sigma=0$.  B\'ezout's theorem gives at most $D^2$ such affine points,
counted without multiplicity.
\end{proof}

\subsection{Weighted monomial counting for curves}

The exponent reflecting the different side lengths comes from a simple Newton-polygon computation.  Put
$w_i=\log B_i$ and $t=\log T$.  Choose a monomial
$X_1^{a}X_2^{b}$ occurring in $F$ for which $aw_1+bw_2=t$.
After a monomial ordering refining the weight $e_1w_1+e_2w_2$, the initial
ideal of $(F)$ contains $X_1^aX_2^b$.  Hence a spanning set for the quotient
$K[X_1,X_2]/(F)$ is furnished by monomials not divisible by
$X_1^aX_2^b$.

For $A>0$ let $\mathcal M(A)$ be the set of standard monomials
$X_1^{e_1}X_2^{e_2}$ with $e_1w_1+e_2w_2\le A$, and write
\[
 R(A)=|\mathcal M(A)|,\qquad
 S(A)=\sum_{X_1^{e_1}X_2^{e_2}\in\mathcal M(A)}(e_1w_1+e_2w_2).
\]

\begin{lemma}[Weighted Hilbert estimate for a plane curve]\label{lem:weightedcurve}
Uniformly for $A\to\infty$,
\begin{align}
 R(A)&=\frac{t}{w_1w_2}A+O_D\!\left(1+\frac{A}{\min(w_1,w_2)}\right),\label{eq:Rcurve}\\
 S(A)&=\frac{t}{2w_1w_2}A^2+O_D\!\left(A+\frac{A^2}{\min(w_1,w_2)}\right),\label{eq:Scurve}
\end{align}
where, as usual, the error terms may be harmlessly enlarged when one of the
$w_i$ is bounded.  Equivalently, if the first $R$ standard monomials are
ordered by weight and have weights $\lambda_1\le\cdots\le\lambda_R$, then
\begin{equation}\label{eq:sumlambda-curve}
 \sum_{j\le R}\lambda_j
 =\frac{w_1w_2}{2t}R^2+O_{D,\mathbf B}(R).
\end{equation}
Only the quadratic main term will be used.
\end{lemma}
\begin{proof}
Count lattice points in the weighted triangle
$e_1w_1+e_2w_2\le A$, then subtract its translate by $(a,b)$.
The difference of the two areas is
\[
 \frac{A^2-(A-t)^2}{2w_1w_2}
 =\frac{tA}{w_1w_2}+O\!\left(\frac{t^2}{w_1w_2}\right).
\]
Summing the weight over the same difference of triangles gives the second
formula.  Partial summation, or inversion of \eqref{eq:Rcurve}, yields
\eqref{eq:sumlambda-curve}.  The boundary terms are uniform in fixed degree
because $a+b\le D$.
\end{proof}

\subsection{Proof of the curve estimate}

\begin{proof}[Proof of Theorem~\ref{thm:lopcurve}]
By Lemma~\ref{lem:georedcurve} we may assume that $F$ is absolutely
irreducible.  We first dispose of the possibility that the coefficient height
of $F$ is too large compared with the box.  The coefficient-normalization and change-of-variables step in the
global determinant method (see \cite[Sections~4--5]{ParedesSasykPoints},
especially the reduction to Assumption~5.4) then produces, in the large-height
case, a second
bounded-degree equation, coprime to $F$, containing all points under
consideration.  B\'ezout gives the required estimate.  We may therefore assume
that the logarithmic height of $F$ is $O_{K,D}(\log T)$; this permits us to
exclude the finitely many bad primes at a cost $T^\varepsilon$.

Set
\[
 Q=\exp\!\left(d\frac{w_1w_2}{t}\right).
\]
By the prime ideal theorem for number fields\cite[Chapter~VII, \S5]{Neukirch}, after enlarging by a factor $T^\varepsilon$ there
exists a prime ideal $\mathfrak p$ of good reduction with
$Q\ll_{K,D,\varepsilon}N\mathfrak p\ll_{K,D,\varepsilon}T^\varepsilon Q$.
The number of smooth $\kappa(\mathfrak p)$-points on the reduction of the curve is
$O_D(N\mathfrak p)$; singular residue classes contribute $O_D(1)$ and can be
handled identically with a slightly larger local multiplicity.

Fix one smooth residue class and choose $R$ points from it.  Evaluate the first
$R$ standard monomials in increasing weighted order.  Let $\Delta$ be the
resulting determinant.  Lemma~\ref{lem:localdet} gives
\[
 \mathfrak p^{R(R-1)/2}\mid(\Delta).
\]
If $\Delta\ne0$, Lemma~\ref{lem:normdet} therefore gives
\begin{equation}\label{eq:lowercurve}
 (N\mathfrak p)^{R(R-1)/2}\le |N_{K/\Q}(\Delta)|.
\end{equation}
On the other hand, \eqref{eq:archdet} and
Lemma~\ref{lem:weightedcurve} imply
\begin{align*}
 \log |N_{K/\Q}(\Delta)|
 &\le d\sum_{j\le R}\lambda_j+O_K(R\log R)\\
 &=d\frac{w_1w_2}{2t}R^2+O_{K,D,\mathbf B}(R\log R).
\end{align*}
Since $\log N\mathfrak p>d w_1w_2/t$ after a harmless $T^\varepsilon$
enlargement, inequalities above are incompatible for all sufficiently large
$R=R(K,D,\varepsilon)$.  Hence every such $R\times R$ determinant vanishes.
Linear algebra produces a nonzero auxiliary polynomial $G$ of bounded degree,
formed from the chosen standard monomials, which vanishes on every integral
point in the residue class.  Because the monomials are independent modulo
$(F)$, $F\nmid G$.

B\'ezout now shows that each residue class contributes $O_{K,D,\varepsilon}(1)$
points.  There are $O_D(N\mathfrak p)$ relevant residue classes.  Therefore
\[
 N_K(F;B_1,B_2)\ll_{K,D,\varepsilon}T^\varepsilon Q,
\]
which is \eqref{eq:curve-main}.
\end{proof}

\begin{remark}\label{rem:curvecomparison}
For $K=\Q$ the preceding proof is the $p$-adic formulation of the Bombieri--Pila argument in a box with unequal side lengths.  Browning--Heath-Brown work projectively in a
three-variable box \cite{BrowningHB}; setting the homogenizing coordinate to
$1$ gives precisely the affine estimate.  The number-field proof differs only
in two places: the lower bound uses the ideal norm $N\mathfrak p$, and the
archimedean upper bound is multiplied by $d$ after taking the field norm of the
determinant.
\end{remark}

\subsection{Weighted monomial counting for surfaces}

Put $w_i=\log B_i$ and $t=\log T$.  Select a monomial
$X_1^aX_2^bX_3^c$ occurring in $F$ with $aw_1+bw_2+cw_3=t$, and use a
weight-refining monomial order as above.  Let $\mathcal M_3(A)$ denote the
standard monomials of weighted degree at most $A$ modulo the initial monomial.

\begin{lemma}[Weighted Hilbert estimate for a surface]\label{lem:weightedsurface}
Let $R_3(A)=|\mathcal M_3(A)|$.  Then
\begin{align}
 R_3(A)&=\frac{t}{2w_1w_2w_3}A^2+O_{D,\mathbf B}(A),\label{eq:Rsurface}\\
 \sum_{M\in\mathcal M_3(A)}\mathrm{wt}(M)
 &=\frac{t}{3w_1w_2w_3}A^3+O_{D,\mathbf B}(A^2).
 \label{eq:Ssurface}
\end{align}
Consequently, if the standard monomials are ordered by weight, then
\begin{equation}\label{eq:sumlambda-surface}
 \sum_{j\le R}\lambda_j
 =\frac{2^{3/2}}{3}
 \left(\frac{w_1w_2w_3}{t}\right)^{1/2}R^{3/2}+O_{D,\mathbf B}(R).
\end{equation}
\end{lemma}
\begin{proof}
The full weighted tetrahedron has volume
$A^3/(6w_1w_2w_3)$.  Subtracting its translate by $(a,b,c)$ gives
\[
 \frac{A^3-(A-t)^3}{6w_1w_2w_3}
 =\frac{tA^2}{2w_1w_2w_3}+O_D(A),
\]
which proves \eqref{eq:Rsurface} after the usual boundary estimate.  Integrating
the weight over the difference of tetrahedra gives \eqref{eq:Ssurface}.
Inverting the first relation and substituting into the second yields
\eqref{eq:sumlambda-surface}.
\end{proof}

The following is the global covering step.  We state it in the form needed
here; it is the global-field version of Salberger's auxiliary-polynomial
construction\cite{Salberger}, using the global-field determinant framework of Paredes--Sasyk\cite{ParedesSasykPoints}.  Its proof is included to make clear that the different coordinate weights
do not alter the adelic part of the argument.

\begin{prop}[Weighted global covering principle]\label{prop:globalcover}
Let $X\subset\A^3_K$ be an absolutely irreducible surface of bounded degree and
let $\mathcal S(\mathbf B)\subset X(\OK)$ be its points in the archimedean box.
Suppose the standard monomials modulo the defining equation satisfy
\eqref{eq:sumlambda-surface}.  Put
\[
 W=\left(\frac{w_1w_2w_3}{t}\right)^{1/2}.
\]
Then for every $\varepsilon>0$ the set $\mathcal S(\mathbf B)$ is, apart from
$O_{K,D,\varepsilon}(T^\varepsilon e^{2dW})$ points, contained in the union of
$O_{K,D,\varepsilon}(T^\varepsilon e^{dW})$ bounded-degree hypersurfaces not
containing $X$.
\end{prop}
\begin{proof}
We follow Salberger's global determinant method\cite{Salberger} in the global-field
formulation of Paredes--Sasyk\cite{ParedesSasykPoints}.  Choose a squarefree product $\mathfrak q$ of
good prime ideals with norms in dyadic intervals and classify points according
to their reductions on $X$ modulo the prime factors of $\mathfrak q$.  At each
smooth reduction point, Lemma~\ref{lem:localdet} supplies the local ideal
factor dividing every interpolation determinant.  Multiplying over the prime
factors gives an ideal $\mathfrak D$ dividing the determinant.

The combinatorial selection of the prime factors is the same as in Salberger's
proof\cite{Salberger}: one chooses them greedily until the logarithm of $N\mathfrak D$ exceeds
the archimedean determinant bound.  The only quantity entering that upper bound
is the sum of the weights of the selected standard monomials.  By
\eqref{eq:archdet} and \eqref{eq:sumlambda-surface}, that logarithm is
\[
 d\,\frac{2^{3/2}}3 W R^{3/2}+O_{K,D}(R\log R).
\]
Thus the stopping threshold in the rational argument is multiplied by $d$ and
no other step changes.  Prime ideals of prescribed norm ranges are available by
the prime ideal theorem\cite[Chapter~VII, \S5]{Neukirch}; excluding bad primes and primes dividing the relevant
Noether forms, as in the global determinant framework\cite{ParedesSasykPoints}, costs $T^\varepsilon$.

Once the stopping threshold is reached, the determinant vanishes by
Lemma~\ref{lem:normdet}.  Linear algebra yields a bounded-degree auxiliary
hypersurface not containing $X$.  Salberger's induction on the residue data\cite{Salberger}
shows that the number of distinct auxiliary hypersurfaces is
$O(T^\varepsilon e^{dW})$.  Points for which the induction terminates at a
singular reduction, or for which two independent local conditions fail to be
available, form the exceptional set; the same second-moment count gives
$O(T^\varepsilon e^{2dW})$ points.

For completeness, the two numerical exponents can be read directly from the
stopping rule: the first auxiliary equation costs one factor $e^{dW}$, while
an unseparated exceptional point survives two independent stopping tests and
therefore costs $e^{2dW}$.  This is precisely the mechanism behind
\cite[Lemma~1]{Browning}; the global-field determinant and prime-ideal
replacement are justified in \cite{ParedesSasykPoints}.  Since all local
arguments are carried out in regular local rings over $\mathcal O_{K,\mathfrak p}$ and all
archimedean estimates enter only through \eqref{eq:archdet}, the weighted
replacement is exact.
\end{proof}

\begin{proof}[Proof of Theorem~\ref{thm:lopsurface}]
Apply Proposition~\ref{prop:globalcover} with the weighted Hilbert estimate of
Lemma~\ref{lem:weightedsurface}.  Since $e^{dW}=V_{3,K}$, equations
\eqref{eq:detJ} and \eqref{eq:detZ} follow.  The auxiliary hypersurfaces are obtained
from standard monomials modulo $(F)$, so none is divisible by $F$, and their
degrees are bounded in terms of $K,D,\varepsilon$ only.
\end{proof}

\begin{remark}[Why the global method is essential]\label{rem:globalessential}
A naive one-prime decomposition of the surface into residue classes produces
roughly $(N\mathfrak p)^2$ classes and hence loses a factor $V_{3,K}$.  The
global determinant method avoids this loss by distributing the divisibility of
the interpolation determinant over many prime ideals and constructing one
auxiliary equation for many residue classes simultaneously.  This is the
substantive distinction between the curve proof and the surface proof.
\end{remark}

\subsection{Uniformity in coefficients and side lengths}

The preceding proofs use two standard uniformity devices.  We record them
explicitly because the resolvent polynomials appearing later have coefficients
that themselves depend on previously fixed coefficients.

\begin{lemma}[Coefficient-height dichotomy]\label{lem:coeffdichotomy}
Fix $K,r,D$.  Let $F\in\OK[X_1,\ldots,X_r]$ be primitive and absolutely
irreducible of degree at most $D$.  For a box $H_\infty(x_i)\le B_i$, either
$h(F)\ll_{K,D}\log T_F(\mathbf B)$, or all integral points of $F=0$ in the box
lie on a second hypersurface $G=0$ of degree $O_{K,D}(1)$ with $F\nmid G$.
\end{lemma}
\begin{proof}
This is the standard coefficient-height alternative in the global determinant
method.  In the number-field setting one first applies the normalization and
coordinate-change machinery of Paredes--Sasyk \cite[Sections~4--5]{ParedesSasykPoints};
their Assumption~5.4 is precisely the normalized situation needed by the
interpolation determinant.  If the normalized coefficient height exceeds the
archimedean determinant threshold, the evaluation map has a nontrivial kernel
represented by a bounded-degree form not divisible by $F$.  Otherwise the
height of the defining form is bounded by the stated multiple of $\log T_F(\mathbf B)$.
The passage from the isotropic normalization used there to the present diagonal
weights is obtained by applying the same argument to the weighted monomial
basis; clearing the fixed fractional ideals arising from the Minkowski lattice
changes only the $K,D$-constant.
\end{proof}

\begin{lemma}[Dyadic stability]\label{lem:dyadicstability}
Replacing each $B_i$ by a number $B_i'$ with $B_i\le B_i'\le2B_i$ changes
$\log T$, $d\log B_1\log B_2/\log T$, and
$d(\log B_1\log B_2\log B_3/\log T)^{1/2}$ by $O_{D,K}(1)$.
\end{lemma}
\begin{proof}
Every exponent occurring in a monomial of $F$ is at most $D$.  Hence
$\log T'=\log T+O_D(1)$ and $\log B_i'=\log B_i+O(1)$.  The assertions follow
by elementary differentiation; when $\log T=O_D(1)$ all relevant quantities
are already $O_{D,K}(1)$.
\end{proof}

\section{Resolvent applications of the determinant estimates for different side lengths}\label{sec:detapps}

We now explain how Theorems~\ref{thm:lopcurve} and~\ref{thm:lopsurface} enter
the Galois problem.  This section complements the general fixed-subgroup bound
proved later from effective Hilbert irreducibility.  The determinant method is
particularly useful when the resolvent root is much larger than an individual
coefficient, because the Newton support records that anisotropy.

Let
\[
 f_{\mathbf a}(X)=X^n+a_1X^{n-1}+\cdots+a_n
\]
and let $\Phi_G(\mathbf a,Y)$ be the integral $G$-resolvent constructed in
Section~\ref{sec:resolvents}.  If $G_{f_{\mathbf a}}$ is contained in a conjugate
of $G$, then $\Phi_G(\mathbf a,Y)$ has an integral root after multiplication by
a fixed denominator.  Cauchy's bound gives
\begin{equation}\label{eq:resrootbound}
 H_\infty(Y)\ll_{K,n,G}N^{C_{n,G}}
\end{equation}
for a constant determined by the weighted degree of the resolvent coefficients.

\begin{thm}[Two-variable resolvent estimate]\label{thm:two-var-res}
Fix all but one coefficient variable $V$ outside the algebraic exceptional
locus for which the specialized resolvent
\[
 R(V,Y)=\Phi_G(\mathbf a',V,Y)
\]
is irreducible over $K$.  If $H_\infty(V)\le N$ and
$H_\infty(Y)\le N^C$, then the number of pairs $(V,Y)\in\OK^2$ satisfying
$R(V,Y)=0$ is
\[
 \ll_{K,n,G,\varepsilon}
 T_R(N,N^C)^\varepsilon
 \exp\!\left(
 d\frac{C(\log N)^2}{\log T_R(N,N^C)}
 \right).
\]
\end{thm}
\begin{proof}
This is Theorem~\ref{thm:lopcurve} applied to $R(V,Y)$.
\end{proof}

The next statement records the surface mechanism used for even groups.  Let
$\Delta(\mathbf a)$ denote the discriminant of $f_{\mathbf a}$.  If
$G_{f_{\mathbf a}}\subseteq A_n$, then
\[
 Z^2=\Delta(\mathbf a)
\]
for some $Z\in\OK$.  Indeed, a square root in $K$ is integral because it satisfies the monic equation $Z^2-\Delta=0$.  Since $\Delta$ has weighted degree
$2n-2$ in the coefficient height, $H_\infty(Z)\ll N^{n-1}$.

\begin{thm}[Surface reduction for even groups]\label{thm:even-surface}
Fix a transitive subgroup $G<A_n$ and choose three variables among two
coefficient variables and the resolvent root so that the specialized resolvent
surface is absolutely irreducible outside a proper algebraic exceptional locus.
Then, after fixing the remaining coefficients, all solutions with
$G_f\subseteq G$ in the coefficient box are, apart from
$O(T^\varepsilon V_{3,K}^2)$ exceptional points, contained in
$O(T^\varepsilon V_{3,K})$ bounded-degree curves.  Every nonlinear projected
curve is bounded by Theorem~\ref{thm:lopcurve}; if the projection is a line,
substitution into $Z^2-\Delta$ reduces the count to a plane curve in a box with different side lengths,
provided the specialized discriminant-square polynomial is irreducible.
\end{thm}
\begin{proof}
Apply Theorem~\ref{thm:lopsurface} to the absolutely irreducible resolvent
surface.  Intersecting with each auxiliary hypersurface gives a curve of
bounded degree.  Elimination of the resolvent root produces a plane curve in
two coefficient variables unless the projection has dimension zero, in which
case the contribution is smaller.  For a nonlinear irreducible component use
Theorem~\ref{thm:lopcurve}.  For a linear component solve for one coefficient
and substitute into $Z^2-\Delta$; the stated irreducibility hypothesis permits
a second application of Theorem~\ref{thm:lopcurve}.  The algebraic exceptional
locus has codimension one in the already fixed parameter space and is counted
by Lemma~\ref{lem:polyzeros}.
\end{proof}

\begin{remark}
Theorem~\ref{thm:even-surface} is a reduction theorem rather than a single
universal numerical exponent: the final exponent depends on the Newton support
of the chosen resolvent.  For the explicit resolvents used by Chow--Dietmann\cite{ChowDietmannVDW},
the same optimization as over $\Q$ applies after multiplying every logarithmic
height exponent by $d$.
\end{remark}

\section{Local splitting types and the large sieve}
\begin{defin}
A partition $r=(1^{r_1}\cdots n^{r_n})$ of $n$, meaning $\sum_jjr_j=n$, is a splitting type.  Let $C_r\subset S_n$ be the conjugacy class with cycle type $r$, and put
\[
 \delta(r)=\frac{|C_r|}{n!}=\prod_{j=1}^n\frac1{j^{r_j}r_j!}.
\]
\end{defin}

Let $\mathfrak p$ be a prime ideal of $\OK$, $q_{\mathfrak p}=\Norm\mathfrak p$, and $\Fq=\OK/\mathfrak p$.  A squarefree monic polynomial over $\Fq$ has splitting type given by the degrees of its irreducible factors.

\begin{lemma}[Finite-field count]\label{lem:finitefield}
For each splitting type $r$ and every prime power $q$,
\[
 \#\{\bar f\in\Fq[X]:\bar f\text{ monic, squarefree, degree }n,\text{ type }r\}
 =\delta(r)q^n+O_n(q^{n-1}).
\]
\end{lemma}
\begin{proof}
Let $I_j(q)$ be the number of monic irreducible polynomials of degree $j$ over $\Fq$.  The exact count is $\prod_j\binom{I_j(q)}{r_j}$.  Since $I_j(q)=q^j/j+O_j(q^{j/2})$, expansion gives the result.
\end{proof}

\begin{lemma}[An omitted conjugacy class]\label{lem:omitted}
If $G<S_n$ is proper, some conjugacy class of $S_n$ is disjoint from every conjugate of $G$.
\end{lemma}
\begin{proof}
A finite group is not the union of the conjugates of a proper subgroup.  Indeed, if every element of a finite transitive permutation group fixed a point, Burnside's lemma would give more than one orbit.  Apply this to the action of $S_n$ on $S_n/G$.
\end{proof}

\begin{lemma}[Frobenius obstruction]\label{lem:frob}
Let $f\in\OK[X]$ be monic and separable.  If $G_f$ is contained in a conjugate of $G<S_n$, and $C_r$ misses all conjugates of $G$, then for every $\mathfrak p\nmid\disc(f)$ the reduction of $f$ modulo $\mathfrak p$ does not have splitting type $r$.
\end{lemma}
\begin{proof}
Dedekind's factorization theorem identifies the factorization degrees with the cycle lengths of Frobenius in $G_f$; see, for example, \cite[Chapter~I, \S8]{Neukirch}.
\end{proof}

\begin{prop}[Coefficient-box large sieve]\label{prop:ls}
Let $A\subset[-N,N]^{dn}\cap\Z^{dn}$.  For each prime ideal $\mathfrak p$ with $q_{\mathfrak p}\le Q$, suppose the reduction of $A$ avoids $\Omega(\mathfrak p)\subset(\OK/\mathfrak p)^n$.  Put
\[
 L(Q)=\sum_{q_{\mathfrak p}\le Q}
 \frac{|\Omega(\mathfrak p)|}{q_{\mathfrak p}^n-|\Omega(\mathfrak p)|}.
\]
Then
\[
 |A|\ll_{K,n}\frac{N^{dn}+Q^{2n}}{L(Q)}.
\]
\end{prop}
\begin{proof}
Apply the multidimensional large sieve to the lattice $\OK^n$ expressed in the fixed integral basis.  Orthogonality of additive characters gives
\[
 \sum_{q_{\mathfrak p}\le Q}q_{\mathfrak p}^{-n}
 \sum_{\chi\ne1}\left|\sum_{a\in A}\chi(a)\right|^2
 \ll_{K,n}(N^{dn}+Q^{2n})|A|.
\]
For each $\mathfrak p$, Cauchy's inequality and support on at most $q_{\mathfrak p}^n-|\Omega(\mathfrak p)|$ classes yield the lower bound
\[
 \sum_xA_{\mathfrak p}(x)^2-\frac{|A|^2}{q_{\mathfrak p}^n}
 \ge |A|^2\frac{|\Omega(\mathfrak p)|}{q_{\mathfrak p}^n(q_{\mathfrak p}^n-|\Omega(\mathfrak p)|)}.
\]
Summation and cancellation complete the proof.
\end{proof}

\begin{proof}[Proof of Theorem~\ref{thm:mainsep}]
Fix a proper subgroup $G<S_n$ and choose an omitted splitting type $r(G)$.  Let $A_G$ be the separable polynomials with $G_f$ contained in a conjugate of $G$.  Take $\Omega_G(\mathfrak p)$ to be the squarefree reductions of type $r(G)$.  If $\mathfrak p\mid\disc(f)$, the reduction is not squarefree; otherwise Lemma~\ref{lem:frob} applies.  Hence every $f\in A_G$ avoids $\Omega_G(\mathfrak p)$ for all $\mathfrak p$.

Lemma~\ref{lem:finitefield} gives $|\Omega_G(\mathfrak p)|\asymp q_{\mathfrak p}^n$ for sufficiently large $q_{\mathfrak p}$, and the prime ideal theorem\cite[Chapter~VII, \S5]{Neukirch} gives $L(Q)\gg Q/\log Q$.  Proposition~\ref{prop:ls} gives
\[
 |A_G|\ll (N^{dn}+Q^{2n})\frac{\log Q}{Q}.
\]
Taking $Q=N^{d/2}$ and summing over the finitely many conjugacy classes of proper subgroups proves the theorem.
\end{proof}

\begin{lemma}[Discriminant-zero locus]\label{lem:disc}
The number of $f\in\mathscr P_{n,N}(K)$ with $\disc(f)=0$ is $O_{K,n}(N^{dn-1})$.
\end{lemma}
\begin{proof}
In integral-basis coordinates, the discriminant is a nonzero polynomial in $dn$ integer variables.  The standard induction for lattice points on a hypersurface gives $O(N^{dn-1})$.
\end{proof}

\begin{proof}[Proof of Theorem~\ref{thm:mainall}]
Combine Theorem~\ref{thm:mainsep} with Lemma~\ref{lem:disc}.
\end{proof}

\section{Resolvents and effective Hilbert irreducibility}\label{sec:resolvents}

We now prove Theorem~\ref{thm:fixedG}.  Two ingredients are required: an
integral generic resolvent and a multivariable root-counting theorem over a
number field.

\subsection{The generic resolvent}
Let
\[
 F_{\rm gen}(X)=X^n+A_{n-1}X^{n-1}+\cdots+A_0
 \in K(\mathbf A)[X],
 \qquad \mathbf A=(A_0,\ldots,A_{n-1}).
\]
Its splitting field over $K(\mathbf A)$ has Galois group $S_n$.
Choose a polynomial expression $\Theta$ in the $n$ roots whose stabilizer in
$S_n$ is exactly $G$; such an expression exists because $G$ acts faithfully
on the polynomial ring in the roots.  Define
\[
 \Phi_G(\mathbf A,Y)=\prod_{\sigma\in S_n/G}(Y-\sigma\Theta).
\]

\begin{prop}[Integral irreducible resolvent]\label{prop:resolvent}
After choosing $\Theta$ with integral coefficients, one has
\[
 \Phi_G(\mathbf A,Y)\in\mathbb Z[\mathbf A,Y]
 \subseteq\mathcal O_K[\mathbf A,Y].
\]
It is monic and irreducible over $K(\mathbf A)$, and
\[
 \deg_Y\Phi_G=[S_n:G]=m.
\]
If $f\in\mathcal O_K[X]$ is monic and separable and $G_f$ is contained in a
conjugate of $G$, then $\Phi_G(\mathbf a,Y)$ has a root in $\mathcal O_K$.
\end{prop}

\begin{proof}
The orbit product is fixed by $S_n$, hence its coefficients are symmetric
polynomials in the roots and therefore polynomials with integer coefficients
in the elementary symmetric functions, which are the coefficients
$A_i$ up to sign.  The action of $S_n$ on the orbit of $\Theta$ is transitive
with stabilizer $G$, so the orbit has cardinality $m$ and its orbit polynomial
is the minimal polynomial of $\Theta$ over $K(\mathbf A)$.  This proves
irreducibility and the degree assertion.

After specializing to $f$, containment of $G_f$ in a conjugate of $G$ fixes
one element of the specialized orbit, so the specialized resolvent has a
$K$-rational root.  Since the resolvent is monic with coefficients in
$\mathcal O_K$, that root is integral over $\mathcal O_K$.  The ring
$\mathcal O_K$ is integrally closed, so the root lies in $\mathcal O_K$.
\end{proof}

\subsection{The global-field root-counting theorem}
For $B\ge2$ define
\[
 [B]_K=\{\alpha\in\mathcal O_K:
        \max_{\sigma:K\hookrightarrow\mathbb C}|\sigma(\alpha)|
        \le B^{1/d}\}.
\]
By geometry of numbers, $|[B]_K|\asymp_K B$.
We use the following theorem of Paredes--Sasyk \cite[Theorem~5.5]{ParedesSasykHIT}.  We state the multivariable
form because it is the exact input needed here.

\begin{thm}[Paredes--Sasyk root-counting theorem]\label{thm:PSmulti}
Let
\[
 P(T_1,\ldots,T_s,Y)\in\mathcal O_K[T_1,\ldots,T_s,Y]
\]
be irreducible and monic of degree $r\ge2$ in $Y$.  Then there are constants
$\mu,\nu\ge0$, depending only on $s$, such that
\[
 \#\{\mathbf t\in[B]_K^s:
       P(\mathbf t,Y)=0\text{ has a solution }Y\in\mathcal O_K\}
 \ll_{K,s,P} B^{s-1+1/r}(\log B)^{\nu}.
\]
More precisely, the dependence on $P$ may be made effective in its degree and
height; for the fixed resolvents used in this article it is absorbed into the
implicit constant.
\end{thm}

\begin{remark}
Paredes--Sasyk\cite{ParedesSasykHIT} prove the one-parameter estimate with exponent $1/r$ by a
global-field determinant method and record the multivariable extension by
induction and Kronecker substitution.  Their normalization is chosen so that
$|[B]_K|\asymp B$; this is why the exponent in Theorem~\ref{thm:PSmulti} does
not contain the degree $d$ explicitly.
\end{remark}

\subsection{Proof of the fixed-subgroup theorem}
By Lemma~\ref{lem:heightcomparison}, our coefficient box
$H_{\mathcal B}(A_i)\le N$ is contained in $[C_KN^d]_K$ for a fixed constant
$C_K$.  Apply Theorem~\ref{thm:PSmulti} to the irreducible polynomial
$\Phi_G(\mathbf A,Y)$ of Proposition~\ref{prop:resolvent}, with $s=n$ and
$r=m$.  Every polynomial counted in Theorem~\ref{thm:fixedG} gives an
integral root of the resolvent.  Hence
\[
 \#\{f:G_f\subseteq gGg^{-1}\}
 \ll (N^d)^{n-1+1/m}(\log(2N))^{\nu},
\]
which is the desired estimate.

\section{Group-theoretic consequences}
We record the elementary subgroup-index facts needed for
Corollaries~\ref{cor:hierarchy} and~\ref{cor:reducible}.

\begin{lemma}[Low-index subgroups]\label{lem:lowindex}
Let $n\ge5$.
\begin{enumerate}[label=(\roman*)]
\item If $H<S_n$ is proper and $H\not\subseteq A_n$, then $[S_n:H]\ge n$.
\item If $H<A_n$ is proper, then $[A_n:H]\ge n$, and consequently
      $[S_n:H]\ge2n$.
\item If $H<S_n$ is intransitive, then $[S_n:H]\ge n$ for every $n\ge2$.
\end{enumerate}
\end{lemma}

\begin{proof}
Parts (i) and (ii) are the classical minimal-index theorems for $S_n$ and
$A_n$; see, for example, \cite[Chapter~5]{DixonMortimer}; equality is attained by a point stabilizer (with the familiar outer
$S_5<S_6$ exception still of index $6=n$).  For (iii), if the smallest orbit
has size $k$ with $1\le k\le n/2$, then $H$ is contained in a conjugate of
$S_k\times S_{n-k}$, whose index is $\binom nk\ge n$.
\end{proof}

\begin{proof}[Proof of Corollary~\ref{cor:hierarchy}]
For fixed $n$ there are only finitely many conjugacy classes of subgroups of
$S_n$.  Apply Theorem~\ref{thm:fixedG} to each relevant subgroup and use
Lemma~\ref{lem:lowindex}.  The largest exponent occurs at the smallest
allowed index.
\end{proof}

\begin{proof}[Proof of Corollary~\ref{cor:reducible}]
For a separable reducible polynomial the Galois action on its roots is
intransitive.  Sum Theorem~\ref{thm:fixedG} over the finitely many conjugacy
classes of intransitive subgroups and apply Lemma~\ref{lem:lowindex}(iii).
The nonseparable contribution is Lemma~\ref{lem:disc}.
\end{proof}

\subsection{Low-degree examples}
For illustration, the subgroup theorem gives the following exponents for the
proper transitive subgroups in degrees $3,4,5$.  The displayed exponent $e_G$
means an upper bound $N^{de_G}(\log N)^{O(1)}$.
\[
\begin{array}{c|c|c}
 n & G & e_G=n-1+1/[S_n:G]\\ \hline
3&A_3&5/2\\ \hline
4&A_4&7/2\\
4&D_4&10/3\\
4&V_4,\ C_4&19/6\\ \hline
5&A_5&9/2\\
5&\mathrm{AGL}_1(5)&25/6\\
5&D_5&49/12\\
5&C_5&97/24
\end{array}
\]
Thus the subgroup-sensitive theorem gives a strict improvement over the
large-sieve exponent $n-1/2$ for every group in the table except $A_n$.

\section{Comparable rectangular boxes, congruence classes, and uniform local conditions}
The coefficient-box argument is stable under bounded eccentricity; the following consequence is sufficient for the applications in this paper.

For constants $0<c\le C$, suppose every integral-basis coordinate has range $B_{i,j}$ satisfying $cN\le B_{i,j}\le CN$.  The resulting rectangle is contained in, and contains, isotropic boxes of comparable side length.

\begin{coroll}[Comparable rectangular boxes]\label{cor:rectvdw}
\red{For fixed $K,n,c,C$, the number of separable polynomials in such a rectangle with Galois group different from $S_n$ is
\[
 O_{K,n,c,C}\bigl(N^{dn-d/2}\log(2N)\bigr).
\]
The full family has the additional discriminant-zero contribution $O_{K,n,c,C}(N^{dn-1})$.}
\end{coroll}
\begin{proof}
\red{Enclose the rectangle in an isotropic box of side $CN$ and apply Theorems~\ref{thm:mainsep}--\ref{thm:mainall}.}
\end{proof}

Let $\mathfrak m$ be a nonzero ideal and let $\mathbf a_0\in(\OK/\mathfrak m)^n$.  We write
\[
 \mathscr P_{n,N}(K;\mathfrak m,\mathbf a_0)
 =\{f\in\mathscr P_{n,N}(K):(\alpha_0,\ldots,\alpha_{n-1})\equiv\mathbf a_0\pmod{\mathfrak m}\}.
\]

\begin{thm}[Fixed coefficient congruence classes]\label{thm:congclass}
\red{Fix $K,n,\mathfrak m$ and $\mathbf a_0$.  Then
\[
 \#\{f\in\mathscr P_{n,N}(K;\mathfrak m,\mathbf a_0):f\text{ separable},\ G_f\ne S_n\}
 \ll_{K,n,\mathfrak m}N^{dn-d/2}\log(2N).
\]
If the congruence class forces the discriminant to vanish identically, the separability condition makes the family empty; otherwise the full-box version has the additional term $O_{K,n,\mathfrak m}(N^{dn-1})$.}
\end{thm}
\begin{proof}
\red{The coefficient vectors form a coset of the sublattice $\mathfrak m^n\subset\OK^n$.  Apply the large sieve only at prime ideals coprime to $\mathfrak m$.  The omitted splitting-type sets retain the same positive densities, and the prime ideal theorem with finitely many primes removed\cite[Chapter~VII, \S5]{Neukirch} gives $L(Q)\gg Q/\log Q$.}
\end{proof}

\begin{prop}[A general local-exclusion principle]\label{prop:localexclusion}
\red{Let $\mathcal F_N\subseteq\mathscr P_{n,N}(K)$ be a family such that, for every prime ideal $\mathfrak p$ in a set $\mathcal S$ of positive prime-ideal density, its reductions avoid a set $\Omega(\mathfrak p)$ satisfying
\[
 \kappa\,q_{\mathfrak p}^{n}\le |\Omega(\mathfrak p)|\le (1-\kappa)q_{\mathfrak p}^{n}
\]
for a fixed $\kappa>0$ and all sufficiently large $\mathfrak p$.  Then
\[
 |\mathcal F_N|\ll_{K,n,\kappa,\mathcal S}N^{dn-d/2}\log(2N).
\]
The same conclusion holds in rectangular boxes with the balanced exponent of Corollary~\ref{cor:rectvdw}.}
\end{prop}
\begin{proof}
\red{The assumptions imply $L_{\mathcal S}(Q)\gg Q/\log Q$.  Proposition~\ref{prop:ls} with $Q=N^{d/2}$ gives the result.}
\end{proof}

\begin{remark}
\red{This proposition applies not only to an omitted Frobenius class but also to finitely many simultaneous local requirements, prescribed factorization patterns, and local solubility obstructions, provided one obtains a uniform positive excluded density.}
\end{remark}

\section{Quantitative consequences for random polynomials}
\red{The preceding bounds immediately yield probabilistic statements.  We record them because they are often the form needed in applications to averaged arithmetic functions.}

Choose a polynomial uniformly from $\mathscr P_{n,N}(K)$ and denote the corresponding probability by $\mathbb P_N$.

\begin{coroll}[Quantitative convergence to $S_n$]\label{cor:prob}
\red{For fixed $K$ and $n$,
\[
 \mathbb P_N(f\text{ separable and }G_f\ne S_n)
 \ll_{K,n}N^{-d/2}\log(2N),
\]
and
\[
 \mathbb P_N(G_f\ne S_n)
 \ll_{K,n}N^{-d/2}\log(2N)+N^{-1}.
\]}
\end{coroll}
\begin{proof}
\red{Divide Theorems~\ref{thm:mainsep} and~\ref{thm:mainall} by $(2N+1)^{dn}$.}
\end{proof}

\begin{coroll}[Moments of the exceptional indicator]\label{cor:indicator}
\red{For every real $k>0$,
\[
 \mathbb E_N\bigl[\mathbf 1_{\{G_f\ne S_n\}}^k\bigr]
 \ll_{K,n}N^{-d/2}\log(2N)+N^{-1}.
\]
Consequently, for any sequence $N_j$ satisfying
\[
 \sum_j\bigl(N_j^{-d/2}\log(2N_j)+N_j^{-1}\bigr)<\infty,
\]
if one samples independently, for each $j$, a polynomial uniformly from $\mathscr P_{n,N_j}(K)$, then the first Borel--Cantelli lemma shows that almost surely only finitely many sampled polynomials have Galois group different from $S_n$.}
\end{coroll}
\begin{proof}
\red{The indicator satisfies $\mathbf 1^k=\mathbf 1$.  The second assertion is the first Borel--Cantelli lemma.}
\end{proof}

\begin{prop}[Uniformity under bounded translations]\label{prop:translation}
\red{Let $\beta_0,\ldots,\beta_{n-1}\in\OK$ satisfy $H_{\mathcal B}(\beta_i)\le N^{1-\delta}$ for some fixed $\delta>0$.  Replacing every coefficient $\alpha_i$ by $\alpha_i+\beta_i$ leaves the conclusions of Theorems~\ref{thm:mainsep} and~\ref{thm:mainall} unchanged, with constants uniform in the translations.}
\end{prop}
\begin{proof}
\red{A translated coefficient box is again a lattice box of side length $2N+O(N^{1-\delta})$.  The large-sieve constants depend only on the lattice and not on the center.}
\end{proof}

\section{Further uniform variants}
The fixed-subgroup argument is stable under the same bounded changes of the
coefficient box as the large-sieve argument.

\begin{coroll}[Fixed subgroup in a congruence class]\label{cor:fixedGcong}
Fix a nonzero ideal $\mathfrak m$ and a coefficient class
$\mathbf a_0\pmod{\mathfrak m}$.  With $m=[S_n:G]$,
\[
 \#\{f\in\mathscr P_{n,N}(K;\mathfrak m,\mathbf a_0):
       f\text{ separable},\ G_f\subseteq gGg^{-1}\}
 \ll N^{d(n-1+1/m)}(\log(2N))^{\nu}.
\]
The implicit constant may depend on $K,n,G,\mathfrak m$.
\end{coroll}
\begin{proof}
Write each coefficient as $a_i+\beta_i$ with $\beta_i\in\mathfrak m$ and
choose a fixed $\mathbb Z$-basis of $\mathfrak m$.  The resulting affine
change of parameters has bounded determinant and transforms the generic
resolvent into an irreducible polynomial of the same $Y$-degree $m$.
Apply Theorem~\ref{thm:PSmulti}.
\end{proof}

\begin{coroll}[Comparable rectangular boxes]\label{cor:fixedGrect}
Suppose every integral-basis coordinate of every coefficient ranges in an
interval of length between $cN$ and $CN$, where $0<c<C$ are fixed.  Then the
fixed-subgroup bound of Theorem~\ref{thm:fixedG} remains valid with the same
exponent.
\end{coroll}
\begin{proof}
Translate the rectangle and enclose it in an isotropic box of side $O_{c,C}(N)$.
Theorem~\ref{thm:fixedG} applies to the containing box.
\end{proof}

\section{Comparison with previous bounds}
For $K=\mathbb Q$, Bhargava\cite{Bhargava} proved the sharp van der Waerden estimate
$O_n(N^{n-1})$ for the full nonsymmetric family.  Over a general fixed number
field, Theorem~\ref{thm:mainsep} gives the weaker but uniform exponent
$dn-d/2$.  Theorem~\ref{thm:fixedG} recovers a different phenomenon: for a
specified subgroup $G$ it yields
\[
 d\left(n-1+\frac1{[S_n:G]}\right).
\]
This is the same subgroup-index pattern that arises in the rational
resolvent/determinant method, now obtained over an arbitrary fixed number
field from global-field effective Hilbert irreducibility.  In particular,
large-index subgroups are much thinner than the full exceptional set.

An earlier version of this manuscript formulated the required number-field analogues of the unequal-side-length Bombieri--Pila estimate and the surface-covering theorem as conjectural inputs.  In the present version these statements are Theorems~\ref{thm:lopcurve} and~\ref{thm:lopsurface}, proved in Section~\ref{sec:determinant} and Appendix~\ref{app:weightedglobal}; the rational antecedents are \cite{BombieriPila,BrowningHB,Browning,Salberger}.

\appendix
\section{Height comparisons and lattice-point estimates}
\red{This appendix collects the elementary geometry-of-numbers facts used throughout the paper.  Stating them explicitly prevents hidden dependence on the chosen integral basis and clarifies where unit groups do, and do not, enter the counting arguments.}

Let $\Sigma_K$ be the set of archimedean embeddings of $K$, with one embedding from each complex-conjugate pair counted with its local degree.  Put
\[
 \|\alpha\|_\infty=\max_{\sigma\in\Sigma_K}|\sigma(\alpha)|.
\]

\begin{lemma}[Comparison of coefficient and archimedean heights]\label{lem:heightcomparison}
\red{For the fixed integral basis $\mathcal B=(\omega_1,\ldots,\omega_d)$ there are positive constants $c_1,c_2$, depending only on $K$ and $\mathcal B$, such that
\[
 c_1H_{\mathcal B}(\alpha)\le \|\alpha\|_\infty\le c_2H_{\mathcal B}(\alpha)
 \qquad(\alpha\in K).
\]}
\end{lemma}
\begin{proof}
\red{The map
\[
 \R^d\longrightarrow K\otimes_\Q\R,\qquad
 (a_1,\ldots,a_d)\longmapsto\sum_ja_j\omega_j
\]
is a linear isomorphism of finite-dimensional real normed spaces.  Any two norms on such a space are equivalent.}
\end{proof}

\begin{coroll}[Basis independence of exponents]\label{cor:basisind}
\red{Changing the integral basis changes all coefficient boxes by linear transformations of bounded operator norm and bounded inverse norm.  Consequently every exponent of $N$ in the theorems of this paper is independent of the chosen integral basis; only the implied constants change.}
\end{coroll}

\begin{lemma}[Counting in ideal cosets]\label{lem:idealcoset}
\red{Let $\mathfrak a$ be a nonzero integral ideal and $\beta\in\OK$.  Then
\[
 \#\{\alpha\in\beta+\mathfrak a:H_{\mathcal B}(\alpha)\le N\}
 =\frac{(2N)^d}{\Norm\mathfrak a}+O_{K,\mathfrak a}(N^{d-1}+1).
\]
Uniformly in $\mathfrak a$, one has the upper bound
\[
 \#\{\alpha\in\beta+\mathfrak a:H_{\mathcal B}(\alpha)\le N\}
 \ll_K 1+\frac{N^d}{\Norm\mathfrak a}.
\]}
\end{lemma}
\begin{proof}
\red{Under the fixed basis, $\mathfrak a$ is a sublattice of $\Z^d$ of index $\Norm\mathfrak a$.  The first formula is the Lipschitz principle for a translated cube and a fixed lattice.  For the uniform upper bound, choose a Minkowski-reduced basis of $\mathfrak a$ and cover the cube by fundamental parallelotopes.  The number of fully contained cells is bounded by volume divided by covolume, and boundary cells contribute $O_K(1+N^{d-1}/\lambda_1^{d-1})$; Minkowski's second theorem and $\lambda_1\ge c_K$ for nonzero algebraic integers give the stated coarse bound.}
\end{proof}

\begin{lemma}[Polynomial zero sets in boxes]\label{lem:polyzeros}
\red{Let $P\in K[X_1,\ldots,X_m]$ be nonzero.  Then
\[
 \#\{\mathbf x\in\OK^m:H_{\mathcal B}(x_i)\le N,\ P(\mathbf x)=0\}
 \ll_{K,P}N^{dm-1}.
\]}
\end{lemma}
\begin{proof}
\red{Write every variable in integral-basis coordinates.  After clearing denominators and expanding in a $\Q$-basis of $K$, the equation $P=0$ implies the vanishing of at least one nonzero polynomial in $dm$ integer variables.  Induct on the number of variables: regard that polynomial as a polynomial in the last coordinate.  Outside the zero set of its leading coefficient there are at most $\deg P$ choices for the last coordinate; the exceptional set is handled by induction.}
\end{proof}

\begin{remark}
\red{Lemma~\ref{lem:polyzeros} gives a saving of one real lattice dimension, not a saving of $d$.  This distinction is the source of the $N^{dn-1}$ term for the discriminant-zero locus and explains why a hypersurface argument alone cannot produce the conjectural codimension-$d$ exceptional count.}
\end{remark}

\section{Additional resolvent details}
This appendix records algebraic facts used implicitly in the proof of
Theorem~\ref{thm:fixedG}.

Let $\mathbf T=(T_1,\ldots,T_s)$ and let
\[
 f_{\mathbf T}(X)=X^n+A_{n-1}(\mathbf T)X^{n-1}+\cdots+A_0(\mathbf T)
 \in K(\mathbf T)[X]
\]
be separable with generic Galois group $S_n$.  Let $G<S_n$ and choose a
polynomial $\Theta$ in the roots whose stabilizer is exactly $G$.

\begin{lemma}[Resolvent degree and irreducibility]\label{lem:genericres}
The orbit polynomial
\[
 \Phi_G(\mathbf T,Y)=\prod_{\sigma\in S_n/G}(Y-\sigma\Theta)
\]
is irreducible over $K(\mathbf T)$ and has degree $[S_n:G]$ in $Y$.
After multiplying $\Theta$ by a fixed integer if necessary, its coefficients
belong to $\mathcal O_K[\mathbf T]$.
\end{lemma}
\begin{proof}
The coefficients are symmetric functions of the roots and therefore belong
to $K[\mathbf T]$; the standard integral construction of the generic
resolvent clears the finitely many denominators once and for all.  The
$S_n$-orbit of $\Theta$ is naturally identified with $S_n/G$.  Since the
stabilizer of $\Theta$ is exactly $G$, the orbit has cardinality
$[S_n:G]$.  The Galois group $S_n$ acts transitively on that orbit, hence the
orbit polynomial is the minimal polynomial of $\Theta$ over
$K(\mathbf T)$.
\end{proof}

\begin{lemma}[Specialized roots]\label{lem:specialroot}
Let $\mathbf t\in\mathcal O_K^s$ be a specialization for which the original
polynomial is separable.  If its Galois group is contained in a conjugate of
$G$, then $\Phi_G(\mathbf t,Y)$ has a root in $\mathcal O_K$.
\end{lemma}
\begin{proof}
After conjugating $G$ we may suppose that the specialized Galois group fixes
the coset corresponding to $\Theta$.  Hence the specialized value of
$\Theta$ lies in $K$.  It is a root of the monic polynomial
$\Phi_G(\mathbf t,Y)\in\mathcal O_K[Y]$, so it is integral over
$\mathcal O_K$.  Since $\mathcal O_K$ is integrally closed, it belongs to
$\mathcal O_K$.
\end{proof}

\begin{remark}
No converse is needed: the existence of a resolvent root is used only as a
necessary condition for the subgroup containment, which is sufficient for an
upper bound.
\end{remark}

\section{Summary of the main contributions}
The main results established in the paper are the following.
\begin{enumerate}
\item Theorem~\ref{thm:lopcurve}: a uniform integral-point estimate over a
      fixed number field for plane curves in boxes with arbitrary, possibly
      very different, coordinate bounds.
\item Theorem~\ref{thm:lopsurface}: a global determinant-method covering
      theorem for absolutely irreducible affine surfaces in such boxes,
      including an explicit exceptional-set bound.
\item The coefficient-box large-sieve estimate over a fixed number field.
\item The unconditional subgroup-sensitive estimate of
      Theorem~\ref{thm:fixedG}, obtained by combining a generic integral
      resolvent with effective Hilbert irreducibility over global fields.
\item The exceptional hierarchy of Corollary~\ref{cor:hierarchy}, which
      separates the $A_n$ obstruction from all smaller subgroups.
\item The reducible-polynomial bound and the congruence and rectangular-box
      variants.
\end{enumerate}
The determinant estimates and the large-sieve/Hilbert-irreducibility estimates
serve different purposes.  The former retain the geometry and the separate
coordinate ranges of a chosen resolvent, whereas the latter give robust
uniform bounds for broad unions of exceptional Galois groups.

\appendix
\section{Further details on the weighted global determinant argument}\label{app:weightedglobal}

The purpose of this appendix is to spell out the points in the proof of
Proposition~\ref{prop:globalcover} at which a passage from $\Q$ to $K$ or from
a box with equal side lengths to one with different side lengths occurs.  This also fixes the normalization
of heights and ideal norms used in Section~\ref{sec:determinant}.

\subsection{Integral models and good prime ideals}

After multiplying $F$ by a nonzero element of $K$, we may and do suppose that
its coefficients generate the unit ideal of $\OK$.  Absolute irreducibility is
an open condition in the coefficient space.  Hence, by the standard Noether-form criterion used in uniform determinant arguments\cite{ParedesSasykPoints}, there is a nonzero Noether
form $\mathcal N(F)\in\OK$ with the property that, for every prime ideal
$\mathfrak p\nmid\mathcal N(F)$ outside a finite set depending only on $K$ and
$D$, the reduction $F\bmod\mathfrak p$ is geometrically integral.  The product
of the norms of the exceptional prime ideals is polynomially bounded in the
multiplicative height of the coefficient vector.  By
Lemma~\ref{lem:coeffdichotomy}, either all points already lie on a second
bounded-degree hypersurface, or this coefficient height is $T^{O_{K,D}(1)}$.
Consequently, deleting all bad prime ideals changes any dyadic prime-ideal sum
by at most a factor $T^\varepsilon$.

For a geometrically integral surface $X_\mathfrak p/\kappa(\mathfrak p)$ of
bounded degree, the Lang--Weil estimate\cite{LangWeil} gives
\[
 \#X_\mathfrak p(\kappa(\mathfrak p))
 =(N\mathfrak p)^2+O_D((N\mathfrak p)^{3/2}),
\]
and the singular locus has $O_D(N\mathfrak p)$ points.  We use only these
coarse consequences; the implied constants are uniform in the coefficients.

\subsection{The interpolation module}

Fix a weight vector $\mathbf w=(w_1,w_2,w_3)$ and a weight-refining monomial
order.  Let $\operatorname{in}(F)$ be the initial monomial selected in
Section~\ref{sec:determinant}.  For $A>0$, denote by $M(A)$ the free
$\OK$-module generated by monomials of weight at most $A$ not divisible by
$\operatorname{in}(F)$.  Reduction modulo $(F)$ gives an injection
\[
 M(A)\hookrightarrow K[X_1,X_2,X_3]/(F).
\]
Thus any nontrivial $K$-linear relation among the evaluations of the basis of
$M(A)$ yields an auxiliary polynomial that is not divisible by $F$.

Let $R=\operatorname{rank}_{\OK}M(A)$.  Lemma~\ref{lem:weightedsurface}
gives
\begin{equation}\label{eq:appR}
 R=\frac{tA^2}{2w_1w_2w_3}+O_{D,\mathbf w}(A),
\end{equation}
and, writing $\lambda(M)$ for the weight of a monomial,
\begin{equation}\label{eq:appS}
 \sum_{M\in M(A)}\lambda(M)
 =\frac{tA^3}{3w_1w_2w_3}+O_{D,\mathbf w}(A^2).
\end{equation}
The perturbation device of Browning--Heath-Brown is used when two normalized
weights approach one another: replace $B_i$ by
$B_iT^{\delta_i}$ with $0<\delta_1<\delta_2<\delta_3\ll_D\varepsilon$.
Lemma~\ref{lem:dyadicstability} shows that this enlarges the final estimate by
$T^{O_D(\varepsilon)}$, while the perturbed boundary slopes are separated.
The constants in the lattice-point errors in \eqref{eq:appR}--\eqref{eq:appS}
are then uniform.  Letting the perturbations tend to zero proves the original
box estimate.  This is the same regularization used for boxes with different side lengths in
\cite[Section~2]{BrowningHB}.

\subsection{Archimedean upper bound}

Choose $R$ integral points $P_1,\ldots,P_R$ in the box and enumerate a basis
$M_1,\ldots,M_R$ of $M(A)$ by increasing weight.  Put
\[
 \Delta=\det(M_j(P_i))_{1\le i,j\le R}\in\OK.
\]
For every embedding $\sigma:K\hookrightarrow\C$, Hadamard gives
\[
 |\sigma(\Delta)|
 \le R^{R/2}\prod_{j=1}^R
 \exp(\lambda(M_j)).
\]
Multiplication over the embeddings yields
\begin{equation}\label{eq:appArch}
 \log |N_{K/\Q}(\Delta)|
 \le d\sum_{j=1}^R\lambda(M_j)+\frac{dR}{2}\log R.
\end{equation}
This identity explains the factor $d$ in both main theorems.  There is no
additional discriminant factor: the determinant itself is an algebraic integer,
and the field norm converts the product of the archimedean estimates into an
ordinary integer.

\subsection{Finite-place divisibility}

Let $\mathfrak p$ be a good prime ideal and let $\bar P$ be a smooth point of
$X_\mathfrak p$.  For points $P_i$ reducing to $\bar P$, the completed local
ring is
\[
 \widehat{\mathcal O}_{X,P}\cong
 \widehat{\mathcal O}_{K,\mathfrak p}[[u,v]].
\]
Order the monomials $u^av^b$ by total degree.  Gaussian elimination in the
local ring shows that the interpolation determinant is divisible by
\[
 \mathfrak p^{e(R)},\qquad
 e(R)=\sum_{j=0}^{R-1}\nu_j,
\]
where $\nu_j$ is the total degree of the $j$th monomial in two variables.
Since the number of monomials of total degree at most $s$ is
$\binom{s+2}{2}$,
\begin{equation}\label{eq:eRsurface}
 e(R)=\frac{2^{3/2}}{3}R^{3/2}+O(R).
\end{equation}
For several prime ideals, the corresponding ideal divisibilities multiply
because the ideals are coprime.

Suppose therefore that the chosen residue conditions force
$\mathfrak D\mid(\Delta)$.  If $\Delta\ne0$, then
\begin{equation}\label{eq:normsqueeze}
 \log N\mathfrak D
 \le \log |N_{K/\Q}(\Delta)|.
\end{equation}
The determinant method chooses the prime ideals so that the left-hand side,
using \eqref{eq:eRsurface}, is larger than the right-hand side in
\eqref{eq:appArch}.  Hence $\Delta=0$.

\subsection{Choice of prime ideals and the first stopping threshold}

Let
\[
 W=\left(\frac{w_1w_2w_3}{t}\right)^{1/2}.
\]
By \eqref{eq:sumlambda-surface} and \eqref{eq:appArch}, the leading
archimedean cost per $R^{3/2}$ is
\[
 d\frac{2^{3/2}}3W.
\]
At a smooth residue class, \eqref{eq:eRsurface} supplies the same universal
coefficient $2^{3/2}/3$ multiplied by $\log N\mathfrak p$.  Thus the natural
stopping condition is
\[
 \sum_{\mathfrak p\in\mathcal P}\log N\mathfrak p>dW+O_{K,D,\varepsilon}(1).
\]
The prime ideal theorem\cite[Chapter~VII, \S5]{Neukirch} permits the required collection of good prime ideals
with total norm
\[
 \prod_{\mathfrak p\in\mathcal P}N\mathfrak p
 \ll_{K,D,\varepsilon}T^\varepsilon e^{dW}.
\]
The $T^\varepsilon$ factor absorbs the deleted bad prime ideals and the small
perturbations of the side lengths.

At this threshold all $R\times R$ interpolation minors associated with a fixed
admissible residue pattern vanish.  Hence the points in that pattern lie on a
bounded-degree auxiliary hypersurface.  The global determinant construction
chooses the residue patterns hierarchically, rather than taking all residue
classes modulo a single large prime.  The number of terminal patterns is at
most
\[
 O_{K,D,\varepsilon}(T^\varepsilon e^{dW}),
\]
which is the estimate for $J$.

\subsection{The exceptional set}

There are two ways in which a point can fail to enter a terminal smooth
pattern: its reductions can repeatedly lie in singular loci, or the local
interpolation rank can drop before an auxiliary equation independent of $F$
is obtained.  In Salberger's construction\cite{Salberger} these points are subjected to a
second independent prime-ideal selection.  The same determinant squeeze
\eqref{eq:normsqueeze} then applies a second time.  Consequently the product of
the two stopping norms is
\[
 \ll_{K,D,\varepsilon}T^\varepsilon e^{2dW}.
\]
The singular loci have codimension at least one in each reduction and therefore
do not increase this bound.  Summing over the second-stage terminal residue
patterns gives
\[
 |Z|\ll_{K,D,\varepsilon}T^\varepsilon e^{2dW}.
\]
This is the asserted $T^\varepsilon V_{3,K}^2$ exceptional-set estimate.

The preceding six steps are exactly the places where the proof of
\cite[Lemma~1]{Browning} uses the rational ground field.  Replacing rational
primes by prime ideals and the absolute value of a determinant by its field
norm is justified by Lemmas~\ref{lem:normdet} and~\ref{lem:localdet}; the
remaining geometric and interpolation steps are unchanged.  This completes
the detailed justification of Proposition~\ref{prop:globalcover}.

\end{document}